\documentclass[12pt,a4paper]{article}
\usepackage[utf8]{inputenc}
\usepackage[T1]{fontenc}
\usepackage{amsmath}
\usepackage{amsthm}
\usepackage{amsfonts}
\usepackage{amssymb}
\usepackage{graphicx}
\usepackage{authblk}
\usepackage{hyperref}
\usepackage{tikz}
\usepackage{wrapfig}
\usepackage{caption}
\usepackage{subcaption}
\usepackage[capposition=bottom]{floatrow}
\usepackage[document]{ragged2e}
\usepackage[left=2.50cm, right=2.50cm, top=2.00cm, bottom=2.00cm]{geometry}

\newtheorem{innercustomthm}{Theorem}
\newenvironment{repthm}[1]
  {\renewcommand\theinnercustomthm{#1}\innercustomthm}
  {\endinnercustomthm}

\newtheorem{theorem}{Theorem}[section]

\newtheorem{proposition}[theorem]{Proposition}
\newtheorem{corollary}[theorem]{Corollary}

\newtheorem{counterexample}[theorem]{Counterexample}
\newtheorem{conjecture}[theorem]{Conjecture}

\newcommand{\bracenom}{\genfrac{\langle}{\rangle}{0pt}{}}

\theoremstyle{definition}
\newtheorem{definition}[theorem]{Definition}

\title{\textbf{LABELED CHIP-FIRING ON BINARY TREES WITH $2^n-1$ CHIPS}}
\author{Gregg Musiker, Son Nguyen}
\affil{University of Minnesota - Twin Cities}
\date{}

\begin{document}
\setlength{\abovedisplayskip}{0pt}
\setlength{\belowdisplayskip}{0pt}
\setlength{\abovedisplayshortskip}{0pt}
\setlength{\belowdisplayshortskip}{0pt}
	
	\maketitle
	
	\begin{abstract}
	\justify
	We study labeled chip-firing on binary trees starting with $2^n-1$ chips initially placed at the root. We prove a sorting property of terminal configurations of the process. We also analyze the endgame moves poset and prove that this poset is a modular lattice.
		
	\end{abstract}
	
	\section{Introduction}\label{sec:intro}
	
	\justify
	Unlabeled chip-firing is the process in which a collection of indistinguishable chips is placed at the nodes of a graph. If a node has at least as many chips as it has neighbors, it can “fire” by sending one chip to each of its neighbors. The process terminates if no site has enough chips to fire, and the configuration when the process terminates is called the terminal configuration or the stable configuration. We refer the readers to \cite{klivans2018mathematics} for a comprehensive study of unlabeled chip-firing; however, we would like to stress the following fundamental result.

    \begin{theorem}[{\cite[Theorem 2.2.2]{klivans2018mathematics}}]\label{thm:confluence} Let $c$ be a configuration.
    
        \begin{enumerate}
            \item \textbf{Local Confluence} Suppose $c_1$ and $c_2$ are two configurations which are both reachable from $c$ after one firing. Then there exists a common configuration $d$ reachable from both $c_1$ and $c_2$ after a single firing.
            \item \textbf{Global Confluence} If a stable configuration $c_s$ is reachable from $c$ after a finite number of legal firing moves, then $c_s$ is the unique stable configuration reachable from $c$.
        \end{enumerate}        
    \end{theorem}
	
	Labeled chip-firing is a variation of unlabeled chip-firing, in which every chip is given an integer label, and additional rules are set to govern which chip moves in each direction during a firing move. Because the chips are now labeled, local confluence may no longer hold, and hence global confluence may not hold either. For example, consider the labeled chip-firing process on an infinite one-dimensional line with $N$ chips labeled $1,\ldots,N$ initially placed at the origin. When a node has at least two chips, it can fire any two chips, sending the smaller chip to the left and the larger chip to the right. However, if one starts with three chips, we instead see that any of the stable configurations in Figure \ref{fig:3chipsPosConfig} can happen.

    \begin{figure}[h!]
        \centering
        \includegraphics[width = 0.8\textwidth]{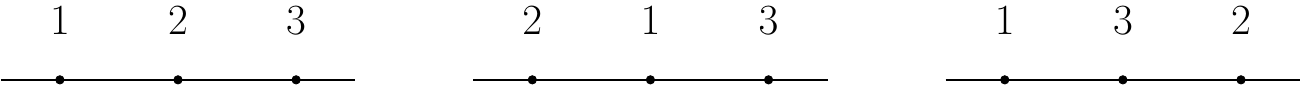}
        \caption{Possible stable configurations with 3 chips}
        \label{fig:3chipsPosConfig}
    \end{figure}
	
	Nonetheless, in some special cases, one still have global confluence despite the lack of local confluence. It was proved in \cite{hopkins2016sorting} and \cite{klivans2022confluence} that the labeled chip-firing process on an infinite one-dimensional line with an even number of chips initially placed at the origin terminates in a unique configuration regardless of the order in which nodes fire and regardless of the choice of chips made at each node. Moreover, in the unique terminal configuration, the chips are in sorted order.
	
    In this paper, we study labeled chip-firing on an infinite binary tree. Specifically, we consider an infinite binary tree with a self-loop at the root and $N=2^n-1$ chips initially placed at the root. The chips are labeled from $1,\ldots,N$.
    We apply the rule such that when a node fires three chips, it sends the smallest chip to its left child, the largest chip to its right child, and the middle chip to its parent. If the root fires, it sends the middle chip to itself (via the self-loop). In this setup, however, global confluence does not hold. For instance, when starting with $7$ chips, both terminal configurations in Figure \ref{fig:7chipsPosConfig} can happen.

    \begin{figure}[h!]
        \centering
        \includegraphics[width = \textwidth]{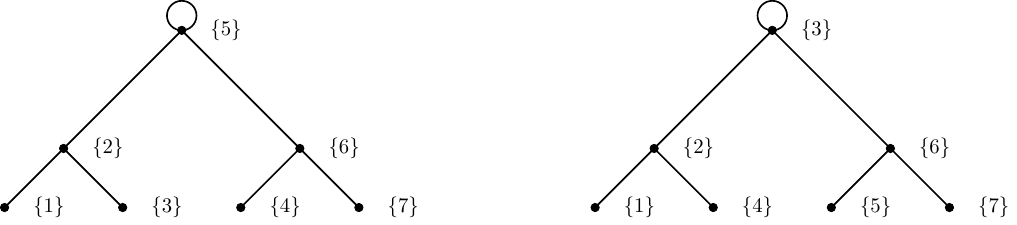}
        \caption{Two possible terminal configurations with $7$ chips}
        \label{fig:7chipsPosConfig}
    \end{figure}
    
    Nevertheless, we do have the following theorem on the sorting property of the terminal configuration:
    
    \begin{theorem} \label{thm:main-thm}
    In labeled chip-firing on an infinite binary tree with $2^n-1$ chips (labeled from $1$ to $2^n-1$) initially placed at the origin, the terminal configuration always has one chip at each node of the first $n$ levels. Moreover, the bottom straight left and right descendants (see Section \ref{sec:def}) of a node contain the smallest and largest chips among its subtree.
    \end{theorem}
    
    Examples of reachable terminal configurations for small cases of $n$ can be seen in Figure \ref{Figure 1}. Consider the node containing chip $11$ in the case of $n=4$, for instance, its bottom straight left and right descendants contain chip $9$ and $15$, which are the smallest and largest chips among its subtree.
    
    After proving Theorem \ref{thm:main-thm} in Sections \ref{sec:unlabeled} and \ref{sec:labeled}, and observing appearances of binary expansions (see Corollary \ref{cor:pattern}) and Eulerian numbers (see Proposition \ref{prop:level-fire}) along the way, we will prove the sorting property for some related versions of binary trees in Section 5. Then we will give some counterexamples for some conjectures generalizing our property in Section \ref{sec:counterex}. In Section \ref{sec:poset}, we will study the poset of "endgame" moves, which will be defined in Section \ref{sec:labeled}. We will show that the poset is actually a modular lattice. Lastly, in Section \ref{sec:conclusion}, we conclude with questions for further research.
    
    \vspace{1em}
    
    \noindent \textbf{Acknowledgements:} The first author was partially supported by NSF grant DMS-1745638, and both authors thank M. Elkin for helpful proofreading, as well as S. Hopkins, C. Klivans, P. Liscio, J. Propp, and the referee for their valuable comments.
    
    \centering
    \begin{figure}[h!]
        \centering
        \includegraphics[width = \textwidth]{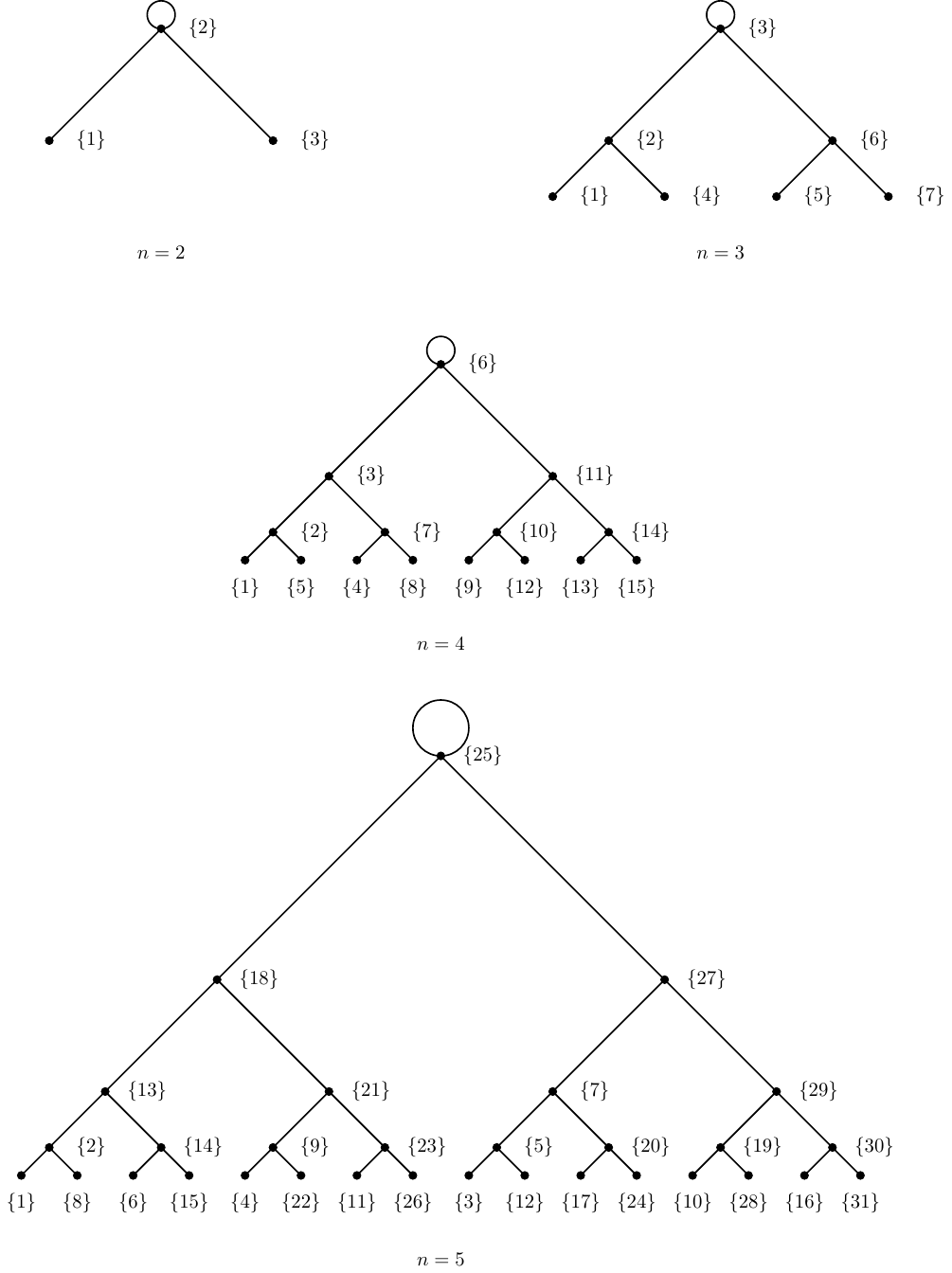}
        \caption{Examples of reachable terminal configurations}
        \label{Figure 1}
    \end{figure} 

	\justify
	\section{Definitions}\label{sec:def}
	
	\justify
	An infinite binary tree is a rooted tree in which every node has two children, which are referred to as the \textit{left child} and \textit{right child}. A node is also referred to as the \textit{parent} of its two children. We label the nodes in the binary tree starting from 1 at the root, going from parent to children and from left to right (see Figure \ref{Figure 2}). Notice that for every node $i$, its two children are labeled $2i$ and $2i+1$, and its parent is labeled $\left[\dfrac{i}{2}\right]$. We also add a self-loop to the root of the tree; thus, during the chip-firing process, every node can fire if and only if it has at least three chips.
    
    \begin{figure}[h!]
        \centering
        \includegraphics[width = \textwidth]{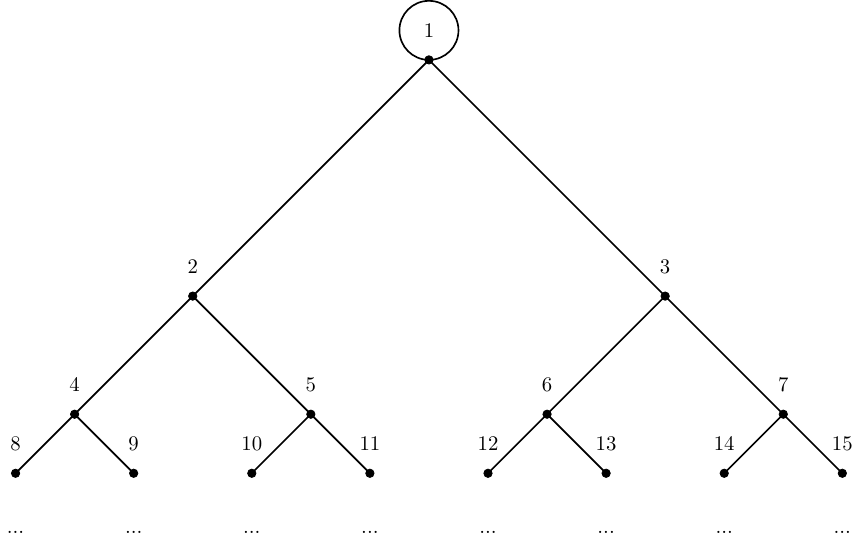}
        \caption{Binary tree and the labeling of the nodes}
        \label{Figure 2}
    \end{figure}

    In the context of an infinite binary tree, we define the rule for labeled chip-firing as follows. We start with $2^n-1$ chips labeled $1$ to $2^n-1$ at the root.
    When a node has at least three chips, it can choose any three chips and fire as follows: it sends the smallest chip to its left child, the largest chip to its right child, and the middle chip to its parent. If the root fires, it sends the middle chip to itself (via the self-loop). The process terminates (or stabilizes) when no node can fire, and we call the configuration when the process terminates the \textit{terminal configuration} or the \textit{stable configuration}.
	
	Now we introduce some more definitions that will be used in our paper. First, we define the \textit{level} of each node $i$, denoted \textit{level}$(i)$, to be dist$(i,1)+1$ where dist$(i,1)$ is the number of edges between node $i$ and node $1$. We also define the \textit{height} of a binary tree to be the maximum level among all of its nodes. For each node $i$, we define the \textit{subtree} of $i$, denote $(i)$, to be the binary tree whose root is $i$. We refer to the nodes in $(i)$ as its \textit{descendants}, and we also refer to $i$ as an \textit{ancestor} of its descendants. We call the nodes $2^ki$, for $k>0$, the \textit{straight left descendants} of $i$, and similarly, we call the nodes $2^{k}(i+1)-1$ its \textit{straight right descendants}. Specifically, we call a node $j$ the \textit{bottom straight left descendant} of node $i$ if, in the terminal configuration, $j$ is the straight left descendant with the bottom-most level that contains some chips. Similarly, we call a node $j$ the \textit{bottom straight right descendant} of node $i$ if, in the terminal configuration, $j$ is the straight right descendant in the bottom-most level that contains some chips. It is worth noting here that the definition of bottom straight descendants depends on the terminal configuration of the process; thus, a node may be a bottom straight descendant if the process starts with some certain number of chips, but not in other cases. We will see this more clearly in the next section. Conversely, we call node $i$ a \textit{straight ancestor} of node $j$ if node $j$ is a straight left or right descendant of node $i$. Lastly, we call node $i$ the \textit{top straight ancestor} of node $j$ if node $i$ is a left (right) child of its parent and node $j$ is a straight right (left) descendant of node $i$. Specifically, the root can be considered as both the top straight ancestor of both its straight left and right descendants.
	
	For example, in Figure \ref{Figure 2}, nodes $2,4,8$ are straight left descendants of node $1$ and nodes $3,7,15$ are straight right descendants. If in the terminal configuration, only the nodes on the first four levels contain some chips, then node $8$ and $15$ are the bottom straight left and right descendants of node $1$. Conversely, node $2$ is the top straight ancestor of nodes $5$ and $11$ but not nodes $4$ and $8$.
 	
    \justify
    \section{Unlabeled Chip-firing}\label{sec:unlabeled}

    Let us first present a brief review of classical unlabeled chip-firing. Let $G=(V,E)$ be a undirected graph and assign a nonnegative integer $C(v)$ to each vertex $v$ in $G$. We say that vertex $v$ has $C(v)$ chips and call $C$ a \textit{chip configuration}. If a vertex $v$ has at least as many chips on it as its degree, i.e. $C(v)\geq \deg{v}$, we say that $v$ is \textit{ready to fire}. When a vertex $v$ fires, it sends chips to its neighbors by sending one chip along each of its incident edges. A chip configuration is \textit{stable} if no vertex is ready to fire, i.e. $C(v) < \deg{v}$ for all $v\in V(G)$. We have the following fundamental result about the chip-firing game.

    \begin{theorem}[\cite{bjorner1991chip}]
        Let $N$ be the total number of chips, i.e. $N = \sum_{v\in V(G)}C(v)$.
        \begin{itemize}
            \item a) If $N > 2|E| - |V|$ then the game is infinite.
            \item b) If $N < |E|$ then the configuration reduces to a unique stable configuration.
            \item c) If $|E|\leq N\leq 2|E|-|V|$ then there exists some initial chip configuration which leads to an infinite process, and some other initial configuration which terminates in finite time.
        \end{itemize}
    \end{theorem}

    Consequently, since our setting is an infinite binary tree with a finite number of chips, we always have $N<|E|$, so our game always terminates. It is worth reminding the reader that in unlabeled chip-firing, Theorem \ref{thm:confluence} is very valuable.

    \begin{repthm}{1.1}[{\cite[Theorem 2.2.2]{klivans2018mathematics}}]
        Let $c$ be a configuration.
    
        \begin{enumerate}
            \item \textbf{Local Confluence} Suppose $c_1$ and $c_2$ are two configurations which are both reachable from $c$ after one firing. Then there exists a common configuration $d$ reachable from both $c_1$ and $c_2$ after a single firing.
            \item \textbf{Global Confluence} If a stable configuration $c_s$ is reachable from $c$ after a finite number of legal firing moves, then $c_s$ is the unique stable configuration reachable from $c$.
        \end{enumerate}        
    \end{repthm}

    Notably, local confluence implies not only global confluence but also the following corollary.

    \begin{corollary}[\cite{klivans2018mathematics}]\label{cor:fire-same-times} Given a configuration $c$ such that $c_s$ is a stable configuration reachable from $c$ after a finite number of legal firing moves, then although the order of legal firing moves is changeable via local confluence, the stabilizing sequence is unique as a multi-set in the following sense:
        \begin{itemize}
            \item The length of any stabilizing sequence is the same, and
            \item The number of times each site fires in any stabilizing sequence is the same.
        \end{itemize}
    \end{corollary}

    Hence, in order to find the length of any stabilizing sequence and the number of times each site fires, it suffices to study one specific sequence. There are other variants of the chip-firing game, including chip-firing on directed graph and chip-firing with a sink vertex. However, these variants are not quite relevant to this paper, so we refer the readers to \cite{bjorner1992chip,holroyd2008chip,klivans2018mathematics}.

    Now we are ready to investigate some properties of the terminal configurations and stabilizing sequences for our setting.
    
    \begin{proposition} \label{prop:height}
         If we start with $N$ chips at the root, where $2^n-1 \leq N \leq 2^{n+1} - 2$, in the terminal configuration, the nodes that contain some chips form a perfect binary tree with height $n$. Furthermore, every node on the same level has the same number of chips.
    \end{proposition}
    
    \begin{proof}
        We will prove the proposition by induction on $n$. In case $n=1$, it is clear that if we start with $1$ or $2$ chips at the root, the configuration is already stable, so the only node that contains some chips is the root, which forms a perfect binary tree with height $1$.
    
        Suppose the proposition is true for $n=k$, we will prove that the proposition is true for $n=k+1$. If we start with $2^{k+1}-1$ to $2^{k+2}-2$ chips at the root, we will fire as follows:
    
    \begin{enumerate}
        \item Fire the root repeatedly until it is not ready to fire anymore.
        \item Alternately fire the left and right branches of the root until they stabilize: after firing a node in the left branch, fire the corresponding node in the right branch.
        \item Whenever the two children of the root fire, fire the root.
    \end{enumerate}
    
        After step $1$, each child of the root will contain the same number of chips. Since every time the root fires, it sends 1 chip to itself, the root always contains at least 1 chip. Therefore, after step $1$, each child of the root contains between $2^k-1$ and $2^{k+1}-2$ chips. Note that step 3 means that whenever the two children of the root fire, they each send 1 chip to the root. Thus, by firing the root immediately after that, the root returns 1 chip back to each child. Hence, when the children of the root fire, they basically send two chips to their children and one chip to themselves. Therefore, the firing process of each branch is similar to firing with $2^k-1$ to $2^{k+1}-2$ initially placed at the root. By the induction hypothesis, we know that after the process stabilizes, in each branch, the nodes that contain some chips form a perfect binary tree with height $n-1 = k$, and every node on the same level has the same number of chips. Since step 2 states that the firing process in both branches is exactly the same, the terminal configuration of the two branches is exactly the same. Thus, after the process stabilizes, the nodes that contain some chips in the whole tree form a perfect binary tree with height $n$, and every node on the same level has the same number of chips. Since we are considering unlabeled chip-firing, the global confluence property applies, so the terminal configuration is always as above. This completes the proof.
    
    \end{proof}
    
    As mentioned in Section 2, the definition of bottom straight descendants depends on the terminal configuration of the process. With Proposition \ref{prop:height}, we now know that if the process starts with $2^n-1$ to $2^{n+1}-2$ chips at the root, the only bottom descendants are the nodes on level $n$, i.e. while some nodes on level $n$ will have a chip, no nodes on level $N$ for $N\geq n+1$ will have chips.
    
    Furthermore, the pattern of which levels have nodes with one chip versus two chips in the terminal configuration is easily determined.
    
    \begin{corollary} \label{cor:pattern}
    If we start with $N$ chips at the root, where $2^n-1 \leq N \leq 2^{n+1} - 2$, then for $0 \leq i \leq n-1$, the resulting terminal configuration has $a_i+1$ chips on each node on level $(i+1)$ where $a_n a_{n-1} \cdots a_2 a_1 a_0$ (with $a_n=1$) is the binary expansion of the number $N+1$.
    \end{corollary}
    
    We will now focus on the case in which we start with $2^n-1$ chips at the root as this is the case we consider in our main theorem. First of all, we have the following special case of Corollary \ref{cor:pattern}.
    
    \begin{corollary} \label{cor:one-chip}
        If we start with $2^n-1$ chips at the root, after the process stabilizes, every node in the first $n$ levels contains exactly 1 chip, and no other node contains any chip.
    \end{corollary}
    
    \begin{proof}
        Since we start with $2^n-1$ chips, we know that the nodes that contain some chips form a perfect binary tree with height $n$. Thus, there are $2^n-1$ nodes that contain some chips, so each must contain exactly 1 chip.
    \end{proof}
    
    Another question that arises naturally in the discussion of chip-firing is how many times each node fires during the firing process. It turns out that in case we start with $2^n-1$ chips at the root, the numbers of times each node fires are Eulerian numbers. Recall that Eulerian number $\bracenom{n}{k}$ is the number of permutations in $S_n$ with $k$ descents. A comprehensive study of Eulerian numbers can be found in Petersen's book (\cite{petersen2015eulerian}), but we will point out that the formula for Eulerian numbers is
    \[ \bracenom{n}{k}=\sum_{i=0}^{k}(-1)^i(k+1-i)^n\binom{n+1}{i}. \]
    Specifically, if $k=1$, we have the following formula for the number of Grassmannian permutations, which are permutations with exactly one descent,
    \[ \bracenom{n}{1}=2^n-(n+1). \]
    Thus, we have the following recursive formula
    \[ \bracenom{n+1}{1}=2^n-1+\bracenom{n}{1}. \]
    Now we will prove that in case we start with $2^n-1$ chips at the root, the numbers of times each node fires are the Eulerian numbers $\bracenom{k}{1}$.
    
    \begin{proposition} \label{prop:level-fire}
        If we start with $2^n-1$ chips at the root, during the firing process, for every $0\leq i<n$, every node on level $n-i$ fires $\bracenom{i+1}{1}$ times.
    \end{proposition}
    
    \begin{proof}
        We will also prove the proposition by induction on $n$. In the case $n=1$, the root does not fire, and since $\bracenom{1}{1}=0$, the proposition is true for $n=1$.
    
        Suppose the proposition is true for $n=k$, we will prove that the proposition is true for $n=k+1$. We will consider the same firing process as in the proof of Proposition \ref{prop:height}. Note that a property of  chip-firing is the number of times each node fires is the same regardless of the order we fire. The firing process of the left and right subtrees are exactly the same as in the case $n=k$. Thus, by the induction hypothesis, we know that for every $0<i<k$ every node on level $k+1-i$ fires $\bracenom{i+1}{1}$ times. Additionally, the root fires $2^k-1$ times in step 1, and $\bracenom{k}{1}$ more times in step 3 (since its children fire $\bracenom{k}{1}$ times each). Thus, the root fires $2^k-1+\bracenom{k}{1}=\bracenom{k+1}{1}$ times. This completes the proof.
    \end{proof}
    
    This leads to the following corollary on the total number of firing moves during the firing process.
    
    \begin{corollary} \label{cor:num-fire}
        If we start with $2^n-1$ chips at the root, during the firing process, the total number of firing moves is $2^{n}(n-3)+n+3$.
    \end{corollary}
    
    The proof of this corollary is straightforward, but we will still verify the calculation here anyway.
    
    \begin{proof}
        Since on level $n-i$ there are $2^{n-i-1}$ nodes, each fires $\bracenom{i+1}{1}=2^{i+1}-(i+2)$ times during the process, the total number of firing moves is
        
        \begin{align*}
            \sum_{i=0}^{n-1}2^{n-i-1}(2^{i+1}-(i+2)) &= \sum_{i=0}^{n-1}(2^n-i2^{n-i-1}-2^{n-i}) \\
            &= n2^n - \sum_{i=0}^{n-1}i2^{n-i-1} - \sum_{i=0}^{n-1}2^{n-i} \\
            &= n2^n-\sum_{k=0}^{n-2}\sum_{i=0}^{k}2^i-(2^{n+1}-2) \\
            &= n2^n-\sum_{k=0}^{n-2}(2^{k+1}-1)-(2^{n+1}-2) \\
            &= n2^n-[2^n-2-(n-1)]-(2^{n+1}-2) \\
            &= 2^n(n-3)+n+3.
        \end{align*}
    \end{proof}
	
    \justify
    \section{Proof of our Main Theorem}\label{sec:labeled}
    
    In this section, we will consider the labeled chip-firing process with $2^n-1$ chips (labeled from 1 to $2^n-1$) initially placed at the root of an infinite binary tree.  Recall our definitions from Section \ref{sec:def}. Let $(i,j)$ be the $j$th-to-last fire of node $i$, we will adopt the relation used by Klivans and Liscio in \cite{klivans2022confluence} and define the relation $(i_1,j_1) > (i_2,j_2)$ if $(i_1,j_1)$ must occur before $(i_2,j_2)$, thereby giving a partial order. Note that $j$ is indexed from 0, which means $(i,0)$ is the last firing move of node $i$. The relation may seem counter-intuitive at first, but later in this section and in Section \ref{sec:poset}, we will prove some interesting results concerning this poset. Particularly, most firing moves at the beginning of the game may occur in any order. Thus, we want to restrict our attention to the bottom of the poset which consists of the "endgame" moves defined as follows.

    \begin{definition}\label{def:endgame-moves}
        We call the moves $(i,j)$ where level$(i)+j<n$ the \textit{endgame moves}. We call the poset formed by these endgame moves with the partial order $(i_1,j_1) > (i_2,j_2)$ if $(i_1,j_1)$ must occur before $(i_2,j_2)$ the \textit{endgame moves poset}.
    \end{definition}
    
    We will later prove that these moves have to occur in some specific order, and the endgame moves posets are symmetric modular lattices. Figure \ref{Figure 11} shows examples of the endgame moves posets for small cases of $n$. First, we will prove that these endgame moves have to occur in some specific order.
    
    \begin{proposition} \label{prop:fire-order}
        Consider a node $i$ with level$(i) = k>1$. Then for all $j$ such that $k+j<n$:
        
        \begin{itemize}
            \item $\left(\left[\dfrac{i}{2}\right],j\right)<\left( i,j\right)<\left(\left[\dfrac{i}{2}\right],j+1\right)$,
            \item $(i,j)$ occurs when node $i$ has 3 chips.
        \end{itemize}
    \end{proposition}
    
    Before proving this proposition, here we present the outline of the proof: 
    
    \begin{enumerate}
        \item $(1,0)<(2,0)$ and $(3,0)$
        \item $(i,0)<(2i,0)$ with the assumption that $\left(\left[\dfrac{i}{2}\right],0\right)<(i,0)$; similarly, we have $(i,0)<(2i+1,0)$
        \item $(2i,0)<(i,1)$ and $(2i+1,0)<(i,1)$
        \item $(1,j)<(2,j)$ with the assumptions $(2,j-1)<(1,j)$ and $(3,j-1)<(1,j)$; similarly, $(1,j)<(3,j)$
        \item $(i,j)<(2i,j)$ with the assumptions $\left(\left[\dfrac{i}{2}\right],j\right)<(i,j)$, $(2i,j-1)<(i,j)$ and $(2i+1,j-1)<(i,j)$; similarly, $(i,j)<(2i+1,j)$
        \item $(i,j)<\left(\left[\dfrac{i}{2}\right],j+1\right)$ assuming both $(2i,j-1)<(i,j)$ and $(2i+1,j-1)<(i,j)$
    \end{enumerate}
    
    The idea of the proof is to do induction on both level$(i)$ and $j$. After steps 1 and 2, we have $(i,0)<(2i,0)$ for all nodes $i$. Then step 3 links $j=0$ to $j=1$ by giving $(2i,0)<(i,1)$ for all nodes $i$. This completes the "base" step, and the "inductive" step continues. First, step 4, analogous to step 1, gives $(1,1)<(2,1)$ and $(1,1)<(3,1)$ using $(2,0)<(1,1)$ and $(3,0)<(1,1)$ proved in step 3. Then step 5, analogous to step 2, completes the case $j=1$ by induction on $i$. Step 6, analogous to step 3, then links $j=1$ to $j=2$. Continuing by using step 4, 5 and 6 repeatedly, we can prove the proposition.
    
    \begin{proof}
        
        \begin{enumerate}

            \item Suppose $(1,0)$ can occur before $(2,0)$, then, after $(1,0)$, node 1 contains at least 1 chip. Also, it will receive at least 1 more chip from $(2,0)$, so it will have at least 2 chips at the end. This is a contradiction since we know at the end every node contains exactly 1 chip. Thus, $(1,0)<(2,0)$, and similarly, $(1,0)<(3,0)$. Also, $(1,0)$ must occur when node $1$ has $3$ chips since otherwise node $1$ will have more than 1 chip at the end.
            
            \item Recall from Section 2 that nodes $2i$ and $2i+1$ are the two children of node $i$. The proof for this case then follows identically to Case 1 above. In particular, suppose $(i,0)$ can occur before $(2i,0)$, then, after $(i,0)$, node $i$ will receive at least 1 chip from $(2i,0)$. It will also receive at least 1 chip from $\left(\left[\dfrac{i}{2}\right],0\right)$ by the assumption that $\left(\left[\dfrac{i}{2}\right],0\right)<(i,0)$. Hence, node $i$ will have at least 2 chips at the end, which is a contradiction. Thus, $(i,0)<(2i,0)$, and similarly, $(i,0)<(2i+1,0)$. Also, since after $(i,0)$, node $i$ will receive at least 1 chip from $\left(\left[\dfrac{i}{2}\right],0\right)$, after $(i,0)$, node $i$ must have no chips left. Therefore, $(i,0)$ must occur when node $i$ has 3 chips.
            
            \item Suppose $(2i,0)$ can occur before $(i,1)$, i.e. there is a node $i$ whose left child fires for the last time before $i$ fires for the second-to-last time. Then, after $(2i,0)$, node $2i$ will receive 2 more chips from node $i$, i.e. from both $(i,1)$ and $(i,0)$. Hence, node $2i$ will have at least $2$ chips at the end, which is a contradiction. Thus, $(2i,0)<(i,1)$, and similarly, $(2i+1,0)<(i,1)$.
            
            \item Suppose $(1,j)$ can occur before $(2,j)$, then, after $(1,j)$, node 1 will keep at least 1 chip and lose $2(j-1)$ more chips. However, since $(3,j-1)<(1,j)$, node 1 will receive at least $j-1$ more chips from node 3 and $j$ more chips from node 2. Hence, node 1 will receive at least $2j-1$ chips. Thus, node 1 will have at least 2 chips at the end, which is a contradiction. Thus, $(1,j)<(2,j)$, and similarly, $(1,j)<(3,j)$. Also, since $(2,j-1),(3,j-1)<(1,j)$, after $(1,j)$, node 1 will lose $2(j-1)$ chips and receive $2(j-1)$ chips. Therefore, $(1,j)$ must occur when node 1 has 3 chips.
            
            \item Suppose $(i,j)$ can occur before $(2i,j)$, then after $(i,j)$, node $i$ will lose $3(j-1)$ more chips. However, since $\left(\left[\dfrac{i}{2}\right],j\right)<(i,j)$, node $i$ will receive at least $j$ more chips from node $\left[\dfrac{i}{2}\right]$. Since $(2i+1,j-1)<(i,j)$, node $i$ will also receive at least $j-1$ more chips from node $2i+1$. Finally, since $(i,j)$ occur before $(2i,j)$, node $i$ will receive at least $j$ more chips from node $2i$. Hence, node $i$ will receive at least $3j-1$ chips, so node $i$ will have at least 2 chips at the end, which is a contradiction. Thus, $(i,j)<(2i,j)$, and similarly, $(i,j)<(2i+1,j)$. Also, since $(2i,j-1),(2i+1,j-1),\left(\left[\dfrac{i}{2}\right],j\right)<(i,j)$, after $(i,j)$, node $i$ will lose $3(j-1)$ chips and receive $3j-2$ chips, so $(i,j)$ must occur when node $i$ has 3 chips.
            
            \item Suppose $(i,j)$ can occur before $\left(\left[\dfrac{i}{2}\right],j+1\right)$, then after $(i,j)$, node $i$ will lose $3(j-1)$ more chips. However, since $(2i,j-1),(2i+1,j-1)<(i,j)$, node $i$ will receive at least $j-1$ more chips from node $2i$ and $2i+1$, and at least $j+1$ more chips from node $\left[\dfrac{i}{2}\right]$. Thus, node $i$ will receive at least $3j-1$ chips, so node $i$ will have at least 2 chips at the end, which is a contradiction. Therefore, $(i,j)<\left(\left[\dfrac{i}{2}\right],j+1\right)$.
            
        \end{enumerate}
        
    \end{proof}

    In general, in the endgame of the process, a node fires its $j$th to last time before its parent's $j$th to last fire but after its parent's $j+1$ to last fire. This is analogous to Lemma 2.6 in \cite{klivans2022confluence}. They showed that for labeled chip-firing on a line, the endgame moves satisfy the relation $(x,y) < (x+1,y)$ and $(x,y) < (x,y+1)$. This gives a nice grid structure on the endgame moves posets. For our case, however, since each node has three neighbors, we do not have such nice grid structure. We will show in Section \ref{sec:poset} that our endgame moves posets are modular lattices but are not distributive. Furthermore, they showed that for labeled chip-firing on a line, the endgame moves occur when the nodes have exactly 2 chips. In our case, the endgame moves occur when the nodes have exactly 3 chips.
            
    An important corollary of this firing relation is the following local confluence property in the endgame.
    
    \begin{corollary} \label{cor:confluence}
        \textbf{(Local confluence at the end)} Suppose $2^n-1$ labeled chips are initially placed at the root, node $1$.  Then, the root fires $\bracenom{n}{1} = 2^n - (n+1)$  times before the terminal configuration.  From the moment $(1,n-2)$ can occur, the order in which the nodes fire does not affect the terminal configuration.
    \end{corollary}
    
    \begin{proof}
        In the end game of the process, a node fires its $j$th to last time before its parent's $j$th to last fire but after its parent's $j+1$ to last fire. Thus, if a node can fire, none of its neighbors can fire (i.e. neither its parent or its two children). Also, each fire occurs when the node has exactly 3 chips. Thus, when a node can fire, it will fire exactly the 3 chips that it contains, and since none of its neighbors can fire, none of its neighbors can affect its firing move. Therefore, the order in which we fire does not affect the terminal configuration.
    \end{proof}
    
    Since the order in which we fire does not matter, we can assume that in the end game, we fire as follows:
    
    \begin{enumerate}
        \item Before $(1,n-2)$, all firing moves $(i,j)$ with level$(i)+j\geq n$ must have occurred. After this step, $(1,n-2)$ is the only possible move, node 1 has 3 chips, and every node on level 2 to $n-1$ has 2 chips.
        \item Fire every node from 1 to $2^{n-1}-1$ exactly once each. After this step, every node on level $n$ has 1 chip, and this chip is the chip it will have when the process stabilizes. Every node on level $n-1$ has no chip, every node on level 2 to $n-2$ has 2 chips and node 1 has 3 chips.
        \item Fire every node from 1 to $2^{n-2}-1$ exactly once each. After this step, every node on level $n$ and $n-1$ has 1 chip, and this chip is the chip it will have when the process stabilizes. Every node on level $n-2$ has no chip, every node on level 2 to $n-3$ has 2 chips and node 1 has 3 chips.
        \item Repeat the above step until the process stabilizes.
    \end{enumerate}
    
    The idea behind the above process is that, in the end game, we will fire in waves, and each wave includes firing every node that can fire once. In each wave, the root will fire and send 2 chips to its children. Every other node that has 2 chips will receive 1 chip from its parent, fire the 3 chips it have and finally receive 2 chips back from its children. The only exceptions are the nodes on the bottom-most level that fires as these nodes will not receive any chip from its children, so they will have no chip after the wave and will only receive 1 more chip from their parents in the next wave.
    
    An example of a wave of fires can be seen in Figure \ref{Figure 4}. First, the root fires and sends chips $4$ and $8$ to its children and sends chip $6$ to itself. Its two children then fire and each sends 1 chip back to the root and 2 chips to the nodes on level 3. The nodes on level 3 then fire, sending chips back to their parents and children. The nodes on level 4 does not fire, so their parents on level 3 does not receive any chip back, and this is the end of the wave. 
    
    \begin{figure}[h!]
        \centering
        \includegraphics[width = 0.55\textwidth]{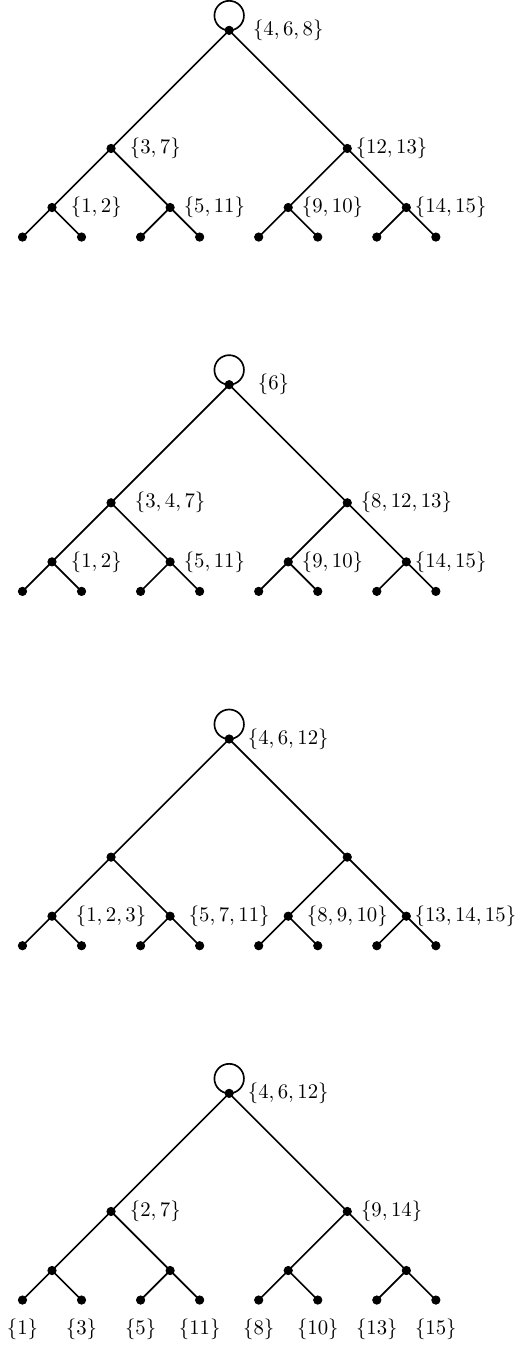}
        \caption{Example of a wave of fires}
        \label{Figure 4}
    \end{figure} 
    
    It can also be seen that as we fire in waves, in each wave, when a chip is sent down from a node to its child, it will be fired again when the child fires. Thus, it will be fired repeatedly until it is fired back from a node to its parent or it reaches the children of the bottom-most level that fires where it will stay for the rest of the process. With these results, we are now ready to prove our main result.
    
    \begin{repthm}{1.2}
        In labeled chip-firing on an infinite binary tree with $2^n-1$ chips (labeled from $1$ to $2^n-1$) initially placed at the origin, the terminal configuration always has one chip at each node of the first $n$ levels. Moreover, the bottom straight left and right descendants of a node contain the smallest and largest chips among its subtree.
    \end{repthm}
    
    \begin{proof}
        The first part of the theorem has already been proved in the previous section. We will now prove the second part. We will prove that the bottom straight left descendant of a node contains the smallest chip among its subtree, and the result for the bottom straight right descendant is analogous. Also, notice that if we want to prove the statement for a node $i$ that is not a right child of its parent, we only have to prove that the statement is true for its top straight ancestor $j$ since they have the same bottom straight left descendant and $(i)\subset(j)$ (here recall that $(i)$ denotes the subtree $(i)$). Thus, we only have to prove the statement for the nodes that are the right child of their parents.
        
        Consider an arbitrary such node $i$, let $c$ be the smallest chip that $i$ contains throughout the firing process. Clearly, $c$ is also the smallest chip the whole subtree $(i)$ contains throughout the firing process; thus, whenever $c$ is fired inside $(i)$, it is fired to the left child. Therefore, from the moment $c$ reaches $i$, it always stays among $i$ and its straight left descendant.
        
        Now note that during the endgame, from the move $(i,n-1-$level$(i))$, i.e. the first move of node $i$ in the endgame, $i$ alternatively sends back and receives 1 chip from its parent. Since $i$ is the right child of its parent, the chip it receives is larger than or equal to the chip it sends back, so $c$ cannot reach $i$ for the first time after $(i,n-1-$level$(i))$. This means that $c$ is already in $(i)$ during the first wave, and since $c$ is only fired to the left child when it is fired, $c$ will reach the end of the first wave, which is the bottom straight left descendant of $i$. This completes the proof.
    \end{proof}
    
    With the sorting property proved for the bottom descendants, a natural question to ask is whether this property can be generalized to nodes on other levels. There are indeed two other types of nodes where this property is true. The first type is the nodes whose top straight ancestor is their parent. Recall the definition of straight ancestors from Section \ref{sec:def}.
    
    \begin{proposition}\label{lemma:type1}
        For any node whose top straight ancestor is its parent, if the node is a left child, then the chip it contains is smaller than both the chips of the parent and the parent's right child and vice versa.
    \end{proposition}
    
    \begin{proof}
        The reason is that when the parent fires the last time, clearly the chip the left child receives is smaller than the chip the right child receives. Also, the parent sends 1 chip back and then receives 1 chip from its parent. Since the parent is a right child, the chip it receives is no smaller than the chip it sends back, and thus is no smaller than the chip the left child contains.
    \end{proof}
    
    The other type is the nodes that are parents of the bottom nodes.
    
    \begin{proposition} \label{lemma:type2}
        The nodes that are parents of the bottom nodes contain the smallest and largest chips among the subtree of their straight ancestors excluding the nodes on the bottom level.
    \end{proposition}
    
    \begin{proof}
         The proof for this case is just slightly more complicated than that for Theorem \ref{thm:main-thm}. Again, we will consider an arbitrary node $i$ that is the right child of its parent, and let $c$ be the second smallest chip $c$ that is sent to its left branch $(2i)$ through out the process. Clearly, inside $(2i)$, when $c$ is fired, it is fired to the left child, except when it is fired together with the smallest chip in which case it is sent back to the parent. Either way, $c$ will remains among $i$ and its straight left descendants, and with the same argument as in Theorem \ref{thm:main-thm}, $c$ is already in $(2i)$ in the second wave, which means $c$ will go to the parent of the bottom straight left descendant. We know that among the chips that are sent to the right branch $(2i+1)$ of node $(i)$, the smallest chip will go to the bottom level, and it is clear that the second smallest chip has to be larger than $c$, so $c$ is the smallest chip among the chips in $(i)$ excluding the nodes on the bottom level. This completes the proof.
    \end{proof}
    
    Examples of the two types can be seen in Figure \ref{Figure 1}. For instance, in the case $n=5$, the chip in node 6 (7) is indeed smaller than that in node 3 (27), and the chip in node 9 (14) is larger than that in node 4 (13). Also in the case $n=5$, we can check that the property is true for the parents of the bottom nodes, which are those on level 4.
    
    \section{Counterexamples}\label{sec:counterex}
    
    With Theorem \ref{thm:main-thm} and Proposition \ref{lemma:type2} proved, it is natural to attempt to generalize the property to nodes on other levels. Unfortunately, the statement for any other node is generally not true. 
    
    \begin{counterexample} \label{ce:ce0}
      The $k$th straight left descendant of node $i$, i.e. node $2^ki$, may not contain the smallest chip among $i$ and its first $k$ levels of descendants.
    \end{counterexample}
    
    For example, in Figure \ref{Figure 1}, the case of $n=5$, node $6$ $(7)$ has a chip with smaller label than that of node $4$ $(13)$ even though node $4$ is the straight left descendant of the root in the binary tree truncated to the top 3 levels, its label is not the smallest in its corresponding subtree.  
    
    From the counterexample, we also know that not only the straight left descendants may not contain the smallest chip among the whole subtree, it even may not contain the smallest chip among the nodes on the same level. Now that we know we cannot extend our main theorem to other nodes, we may want to try to study a more local conjecture. Our local conjecture is: in the terminal configuration, for every node $i$ not on the bottom-most level, its chip is larger than the chip of its left child's and is smaller than that of its right child's. Unfortunately, the conjecture is also not true, and we will show a counterexample here.
    
    \begin{counterexample} \label{ce:ce1}
        For some node $i$, its chip may be smaller than the chip of its left child's or larger than the chip of its right child's.
    \end{counterexample}
    
    To make the description of the counterexample simpler, we will consider a colored chip-firing process with red and blue chips similar to that discussed in Klivans' book (\cite[Sec. 5.4]{klivans2018mathematics}). Colored chip-firing is a variation of labeled chip-firing where each chip is colored red or blue. When three chips of the same color fire, we send one to the left child, one to the right child, and one to the parent. When three chips of different colors fire, we send a red chip to the left child, a blue chip to the right child, and the other chip to the parent. Suppose we have a sorting property for labeled chip-firing, we say the property is true for colored chip-firing if for every node $i,j$ such that node $i$ must contain a smaller chip than node $j$ in labeled chip-firing, node $j$ cannot contain a red chip while node $i$ contain a blue chip in colored chip-firing. We have the following proposition about the equivalence between colored chip-firing and labeled chip-firing from Klivans' book (\cite[Sec. 5.4]{klivans2018mathematics}).

    \begin{proposition} \label{prop:colored}
        A sorting property is true for labeled chip-firing if and only if it is true for all possible colorings of colored chip-firing.
    \end{proposition}

    Now we will show a counterexample in colored chip-firing for counterexample \ref{ce:ce1}. First, let us quickly introduce a notation that will be useful. Denote $[i,XYZ]$, where $X,Y,Z \in\{R,B\}$, to be "fire 3 chips colored $X,Y,Z$ at node $i$" with $R$ for red and $B$ for blue. Thus, $[1,RRB]$ means "fire 2 red chips and 1 blue chips at node $1$". Furthermore, we will use $[i,XYZ]^n$ to say "apply $[i,XYZ]$ $n$ times".

    Now consider colored chip-firing with 63 chips, in which 23 chips are colored red and 40 chips are colored blue, initially placed at the root. We will first fire the red chips as follows:\\
    
    $[1,RRR]^{11}$ $\longrightarrow$
    $[2,RRR]^3$ $\longrightarrow$
    $[3,RRR]^3$ $\longrightarrow$
    $[4,RRR]$ $\longrightarrow$
    $[5,RRR]$ $\longrightarrow$
    $[6,RRR]$ $\longrightarrow$
    $[7,RRR]$ $\longrightarrow$
    $[1,RRR]^3$ $\longrightarrow$
    $[2,RRR]^2$ $\longrightarrow$
    $[3,RRR]^2$ $\longrightarrow$
    $[1,RRR]^2$ $\longrightarrow$
    $[2,RRR]$ $\longrightarrow$
    $[3,RRR]$ $\longrightarrow$
    $[4,RRR]$ $\longrightarrow$
    $[5,RRR]$ $\longrightarrow$
    $[6,RRR]$ $\longrightarrow$
    $[7,RRR]$ $\longrightarrow$
    $[1,RRR]$ $\longrightarrow$
    $[2,RRR]$ $\longrightarrow$
    $[3,RRR]$ $\longrightarrow$
    $[1,RRR]$\\
    
    \begin{figure}[h!]
        \centering
        \includegraphics[width = 0.9\textwidth]{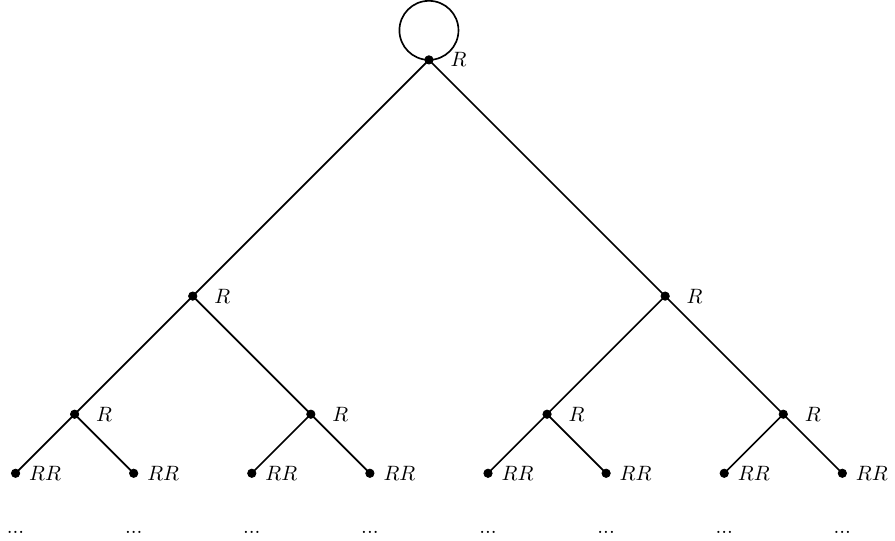}
        \caption{Counterexample \ref{ce:ce1} - red chips arrangement (The 40 blue chips are still at the root)}
        \label{Figure 7}
    \end{figure}
    
    At this moment, the red chips are arranged like in Figure \ref{Figure 7}. Then we will fire wisely so that in the subtree $(2)$, we fire as many red chips as possible out of the straight left descendants. In the subtree $(3)$, on the other hand, we fire the blue chips without touching the red chips. Specifically, we will fire as follows:\\
    
    $[1,RBB]$ $\longrightarrow$
    $[1,BBB]^{19}$ $\longrightarrow$
    $[2,RBB]^2$ $\longrightarrow$
    $[1,BBB]$ $\longrightarrow$
    $[4,RRR]$ $\longrightarrow$
    $[8,RRR]$ $\longrightarrow$
    $[2,RBB]$ $\longrightarrow$
    $[2,BBB]^4$ $\longrightarrow$
    $[1,BBB]$ $\longrightarrow$
    $[2,BBB]$ $\longrightarrow$
    $[1,BBB]^2$ $\longrightarrow$
    $[4,RBB]^2$ $\longrightarrow$
    $[2,BBB]$ $\longrightarrow$
    $[5,RBB]$ $\longrightarrow$
    $[5,BBB]^2$ $\longrightarrow$
    $[2,BBB]$ $\longrightarrow$
    $[4,BBB]$ $\longrightarrow$
    $[9,RBB]$ $\longrightarrow$
    $[10,RBB]$ $\longrightarrow$
    $[5,BBB]$ $\longrightarrow$
    $[2,BBB]$ $\longrightarrow$
    $[11,RBB]$ $\longrightarrow$
    $[10,RRB]$ $\longrightarrow$
    $[11,RBB]$ $\longrightarrow$
    $[5,RBB]$ $\longrightarrow$
    $[1,BBB]$ $\longrightarrow$
    $[3,BBB]^8$ $\longrightarrow$
    $[1,BBB]$ $\longrightarrow$
    $[2,BBB]$ $\longrightarrow$
    $[4,BBB]$ $\longrightarrow$
    $[8,RBB]$ $\longrightarrow$
    $[9,RBB]$ $\longrightarrow$
    $[1,BBB]^4$ $\longrightarrow$
    $[2,BBB]$ $\longrightarrow$
    $[4,BBB]$ $\longrightarrow$
    $[5,BBB]$ $\longrightarrow$
    $[2,BBB]$ $\longrightarrow$
    $[1,BBB]$ $\longrightarrow$
    $[3,BBB]^2$ $\longrightarrow$
    $[1,BBB]$ $\longrightarrow$
    $[2,BBB]$ $\longrightarrow$
    $[6,BBB]^3$ $\longrightarrow$
    $[7,BBB]^3$ $\longrightarrow$
    $[3,BBB]^2$ $\longrightarrow$
    $[1,BBB]$ $\longrightarrow$
    $[3,BBB]$ $\longrightarrow$
    $[1,BBB]$ $\longrightarrow$
    $[6,BBB]$ $\longrightarrow$
    $[7,BBB]$ $\longrightarrow$
    $[3,BBB]$ $\longrightarrow$
    $[12,BBB]$ $\longrightarrow$
    $[13,BBB]$ $\longrightarrow$
    $[14,BBB]$ $\longrightarrow$
    $[15,BBB]$ $\longrightarrow$
    $[6,BBB]$ $\longrightarrow$
    $[7,BBB]$ $\longrightarrow$
    $[3,RBB]$ $\longrightarrow$
    $[6,RRB]$ $\longrightarrow$
    $[7,RBB]$ $\longrightarrow$
    $[12,RBB]$ $\longrightarrow$
    $[13,BBB]$ $\longrightarrow$
    $[14,RBB]$ $\longrightarrow$
    $[15,BBB]$\\
    
    Thus, before the endgame, the subtree $(2)$ will look like Figure \ref{Figure 8} while subtree $(3)$ look like Figure \ref{Figure 8b}. Now we will fire in waves. In the first wave, the red chip in node $8$ will be sent to node $16$, and thus when node $16$ fires, there will be only 2 red chips remaining among node $2$ and its straight left descendants. Hence, after 3 waves, nodes $32$ and $16$ will contain a red chip, but node $8$ will contain a blue chip. On the other hand, in the subtree $(3)$ after the first wave, nodes $12, 13, 14, 15$ will each send 1 red chip back to their parents. Then after the second wave, nodes $6,7$ will each send 1 red chip back to node $3$, so after the third wave, node $3$ will send 1 red chip back to node $1$. Hence, after the fourth wave, node $4$ will contain a red chip, which is a counterexample since node $8$ contains a blue chip (see Figure \ref{Figure 8c}).
    
    \begin{figure}[h!]
        \centering
        \includegraphics[width = 0.9\textwidth]{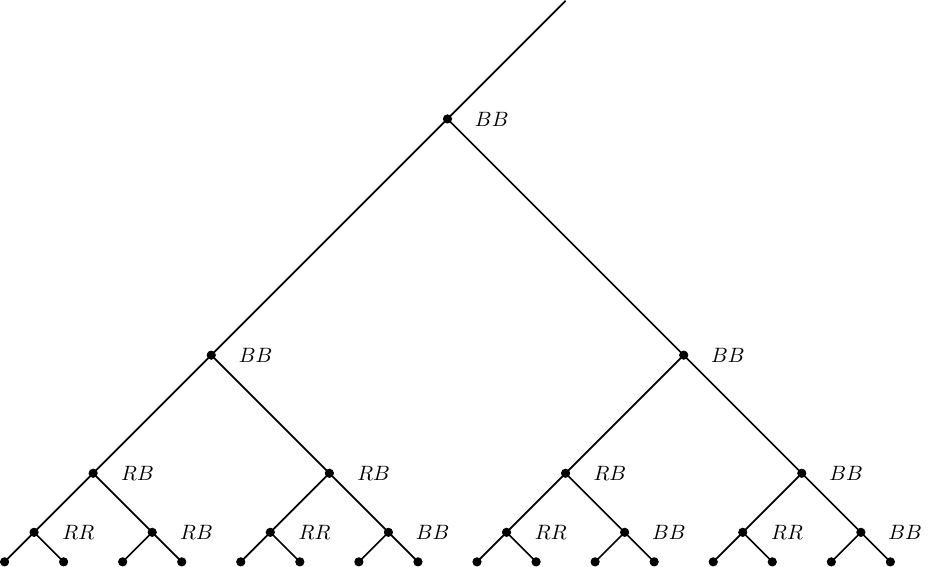}
        \caption{Counterexample \ref{ce:ce1} - subtree $(2)$ before the endgame}
        \label{Figure 8}
    \end{figure}
    
    \begin{figure}[h!]
        \centering
        \includegraphics[width = 0.9\textwidth]{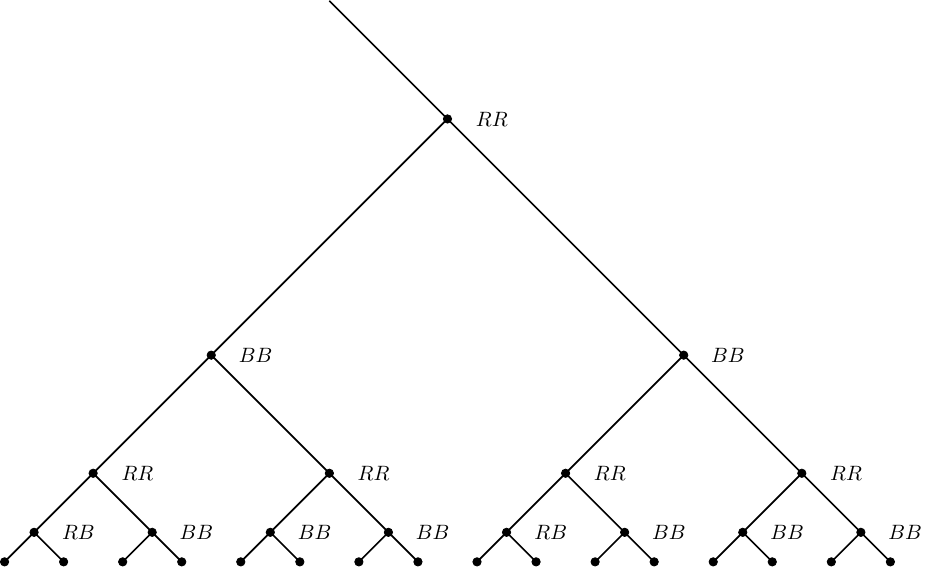}
        \caption{Counterexample \ref{ce:ce1} - subtree $(3)$ before the endgame}
        \label{Figure 8b}
    \end{figure}
    
    \begin{figure}[h!]
        \centering
        \includegraphics[width = 0.9\textwidth]{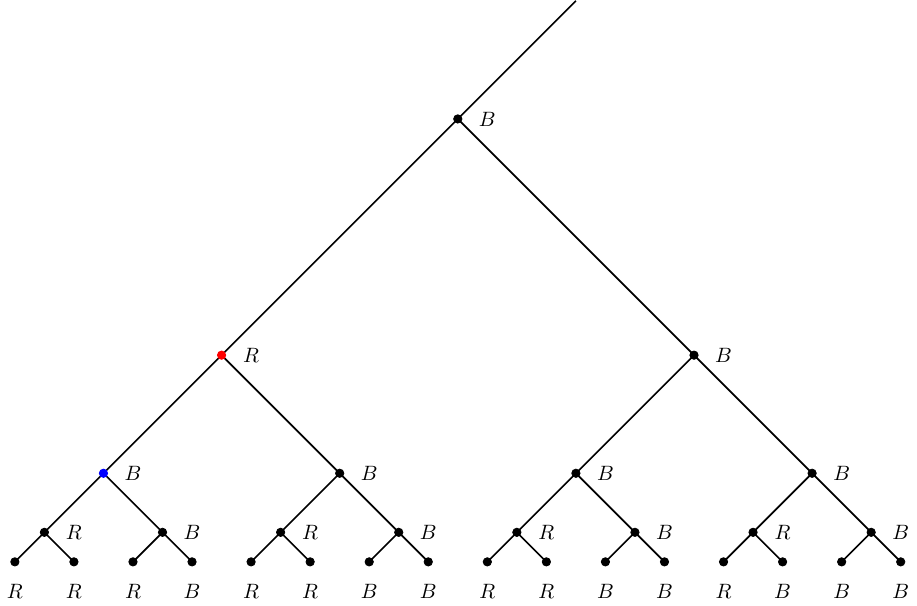}
        \caption{Counterexample \ref{ce:ce1} - subtree $(2)$ of the terminal configuration}
        \label{Figure 8c}
    \end{figure}
    
    We can label the 23 red chips from 1 to 23 and the blue chips from 24 to 63 to verify that this is indeed a counterexample for labeled chip-firing.

    \section{Firing Move Poset}\label{sec:poset}
    
    In this section, we will study the poset of "endgame" moves. Recall that in Section 4, we defined "endgame" moves in a firing process with $2^n-1$ chips to be the moves $(i,j)$ where level$(i)+j<n$. In other words, the last $j$ instances of firing at node $i$, as long as this inequality hold. We also recall the relation $(i_1,j_1)>(i_2,j_2)$ if $(i_1,j_1)$ must occur before $(i_2,j_2)$ and proved that $\left(\left[\dfrac{i}{2}\right],j\right)<\left( i,j\right)<\left(\left[\dfrac{i}{2}\right],j+1\right)$.
    
    Now let us give a rigorous definition for our poset. 
    
    \begin{definition}\label{def:P_n}
        Define the partial ordered set $P_n$ to be a set whose elements are the moves $(i,j)$ where level$(i)+j<n$, together with the binary relation $\leq$ where $(i_2,j_2)\leq(i_1,j_1)$ if $(i_1,j_1)$ must not occur after $(i_2,j_2)$.
    \end{definition}
    
    The Hasse diagrams of $P_4$ and $P_5$ can be seen in Figure \ref{Figure 11}. From the definition of $P_n$, it is clear that $P_1\subset P_2\subset P_3 \subset...$. Let us recall some terminology common in enumerative combinatorics, e.g. see \cite[Chapter 3]{stanley2011enumerative}. We say that a poset $P$ is a \textit{lattice} if every pair of elements $s,t$ has a unique least upper bound (called a \textit{join}, denoted $s \vee t$) and a unique greatest lower bound (called a \textit{meet}, denoted $s \wedge t$). $P$ is \textit{graded} if there is a rank function $\rho: P \rightarrow \mathbb{N}$ such that $\rho(s) = 0$ for every minimal element $s$, and $\rho(t) = \rho(s) + 1$ if element $t$ covers $s$, meaning $s \leq t$ with no elements in between.  This is often denoted as $s \lessdot t$. Note that when $P$ is graded, its rank function is also unique. A lattice is \textit{modular} if it is graded and its rank function $\rho$ satisfies
    \[ \rho(s) + \rho(t) = \rho(s \vee t) + \rho(s \wedge t) \]
    for all $s,t$.
    
    \begin{figure}[h!]
        \centering
        \includegraphics[width = 0.9\textwidth]{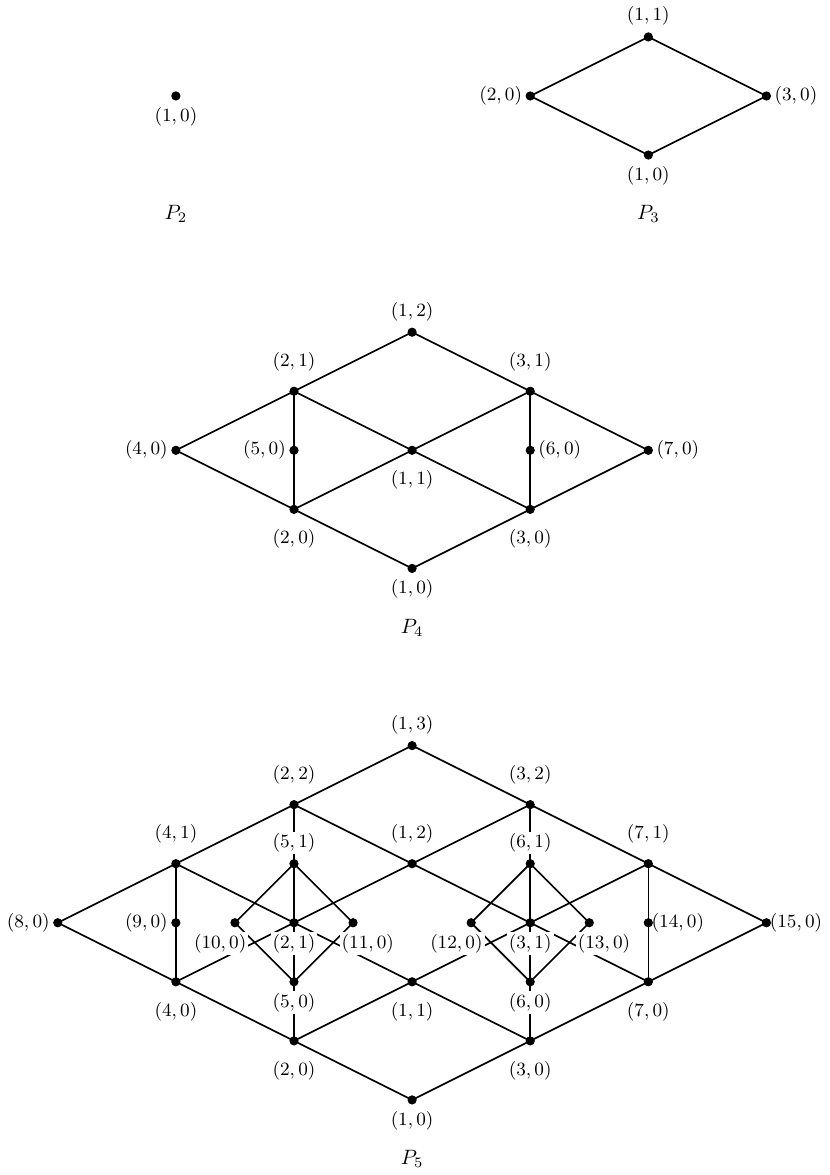}
        \caption{Hasse diagrams of the posets $P_2,P_3,P_4,P_5$}
        \label{Figure 11}
    \end{figure}
    
    We will first show that the posets $P_n$ are lattices.
    
    \begin{proposition} \label{prop:lattice}
        $P_n$ is a lattice.
    \end{proposition}
    
    \begin{proof}
        Consider any pair of elements $(u,p)$ and $(v,q)$, we will prove that their join exists, and the argument for their meet follows analogously.
        
        Since $u$ and $v$ are nodes of a binary tree, there exists a unique path from $u$ to $v$. Suppose the path is $u=a_0-a_1-a_2-...-a_i=w=b_j-b_{j-1}-...-b_1-b_0=v$ with $i,j\geq0$ and $a_k,b_k$ are children of $a_{k+1},b_{k+1}$ for all $k$. For now, let us temporarily remove the restriction level$(i)+j<n$ for elements $(i,j)\in P_n$ and suppose that every element $(i,j)$ is in $P_n$ and Proposition \ref{prop:fire-order} is true for all elements. This temporary inclusion is only for our sake of convenience and we will show that it will not affect our argument. With this inclusion, however, we can now form the chains $C_1:(u,p)\lessdot(a_1,p+1)\lessdot...\lessdot(w,p+i)\lessdot(b_{j-1},p+i)\lessdot...\lessdot(v,p+i)$ and $C_2:(v,q)\lessdot(b_1,q+1)\lessdot...\lessdot(w,q+j)\lessdot(a_{i-1},q+j)\lessdot...\lessdot(u,q+j)$ without any concern that some elements may not be in $P_n$. It is worth noticing that in the $C_1$, every element $(a,b)$ is the element of type $(a,k)$ with the lowest value of $k$ such that $(a,k)\geq(u,p)$. In other words, every move $(a,b)$ in the chain is the latest move of node $a$ that must not occur after $(u,p)$. The same applies for $C_2$.
        
        If $p+i\leq q$ then $(u,p)<(v,q)$, so their join is simply $(v,q)$. Similarly, if $q+j\leq p$ then $(v,q)<(u,p)$, so their join is simply $(u,p)$. Suppose none of the above inequalities is true, then since $p+i>q$, we have $(v,p+i)>(v,q)$. Also, since none of the above inequalities is true, $(u,p)$ and $(v,q)$ are incomparable, there is an element in $C_1$ such that it is larger than $(v,q)$ but every element before it is not comparable with $(v,q)$. Suppose this element is $(a_k,p+k)$ for some $0<k\leq i$, we will see shortly that if this element is $(b_k,p+i)$ for some $0<k\leq j$ then we can consider the chain $C_2$ and the argument follows analogously.
        
        First, we will show that $p+k=q+j$. It is easy to see that $p+k\geq q+j$ since $(a_k,q+j)$ is the latest move of node $a_k$ that must not occur after $(v,q)$. On the other hand, if $p+k>q+j$, then $p+k-1\geq q+j$, so $(a_{k-1},p+k-1)\geq(a_{k-1},q+j)>(v,q)$, which is a contradiction since we assume $(a_{k-1},p+k-1)$ is incomparable with $(v,q)$. Thus, $p+k=q+j$, so $(a_k,p+k)$ is also in $C_2$.
        
        Second, we will show that $(a_k,p+k)=(a_k,q+j)$ plays the same role in $C_2$, that is, $(a_k,q+j)$ is the first element that is larger than $(u,p)$ in $C_2$. Clearly, $(a_k,q+j)>(u,p)$ since it is an element of $C_1$. Also, if $k<i$ then $(a_{k+1},q+j)<(a_{k+1},p+k+1)$, which is the latest move of node $a_{k+1}$ that must not occur after $(u,p)$, so $(a_{k+1},q+j)$ is incomparable with $(u,p)$. If $k=i$ then similarly, $(b_{j-1},p+i-1)<(b_{j-1},p+i)$, which is the latest move of node $b_{j-1}$ that must not occur after $(u,p)$, so $(b_{j-1},p+i-1)$ is incomparable with $(u,p)$. Thus, $(a_k,p+k)=(a_k,q+j)$ is the first element that is larger than $(u,p)$ in $C_2$.
        
        Next, we will show that $(a_k,p+k)$ is the join of $(u,p)$ and $(v,q)$. Clearly, $(a_k,p+k)$ is an upper bound of $(u,p)$ and $(v,q)$. Consider an arbitrary upper bound $(s,t) \neq (a_k,p+k)$, we will prove that $(s,t)>(a_k,p+k)$. Clearly, if $s=a_k$ then $t>p+k$, so $(s,t)>(a_k,p+k)$. If $s$ is one of the $a_0,a_1,...,a_{k-1}$, then if $t<q+j$, $(s,t)$ is either smaller than or incomparable to $(v,q)$, which is a contradiction. If $t\geq q+j$ then $(s,t)>(a_k,t)\geq(a_k,q+j)$. The analogous argument applies in case $s$ is one of the $a_{k+1},...,w,...,b_0$. Lastly, if $s$ is not any of the nodes $a_0,...,w,...,b_0$, then there exists two chains $(s,t)\gtrdot(c_1,d_1)\gtrdot...\gtrdot(u,d_l)\geq(u,p)$ and $(s,t)\gtrdot(e_1,f_1)\gtrdot...\gtrdot(v,f_m)\geq(v,q)$. Since our graph is a tree, the paths $s-c_1-..-u$ and $s-e_1-...-v$ must merge at some point, i.e. there exist $r$ such that $c_1=e_1,c_2=e_2,...,c_r=e_r$ and $c_r=e_r$ is one of the nodes $a_0,...,w,...,b_0$. Thus, $(s,t)>(c_r,d_r)\geq(a_k,p+k)$. Hence, $(a_k,p+k)$ is the join of $(u,p)$ and $(v,q)$.
        
        Finally, we will show that the above temporary inclusion of $P_n$ does not affect our result. This is because $(1,n-2)$ is the $\hat{1}$ element of $P_n$, so $(1,n-2)$ is an upper bound of $(u,p)$ and $(v,q)$. Thus, $(1,n-2)>(a_k,p+k)$, which means $(a_k,p+k)$ is also an endgame moves and thus is in the original $P_n$. This completes the proof.
    \end{proof}
    
    Next, we define the following rank function $\rho((i,j))=$level$(i)+2j$. We will prove that $\rho$ makes $P_n$ a graded poset.
    
    \begin{proposition} \label{prop:graded}
        $P_n$ equipped with the rank function $\rho((i,j))=$level$(i)+2j$ is graded.
    \end{proposition}
    
    \begin{proof}
        This is trivial from the fact that for every element $(i,j)\in P_n$, there are only three possible elements covering $(i,j)$: $(2i,j),(2i+1,j)$ and $\left(\left[\dfrac{i}{2}\right],j+1\right)$. It is easy to check that the rank of these elements are all $\rho((i,j))+1$.
    \end{proof}
    
    Now we are ready to state our second theorem.
    
    \begin{theorem} \label{thm:modular}
        $P_n$ is a symmetric modular lattice.
    \end{theorem}
    
    \begin{proof}
        The fact that $P_n$ is a lattice and is graded has already been established, so we now show that $P_n$ is modular. Consider two arbitrary elements $(u,p)$ and $(v,q)$, we will show that $\rho((u,p))+\rho((v,q))=\rho((u,p)\vee(v,q))+\rho((u,p)\wedge(v,q))$. Reusing the notations in Proposition \ref{prop:lattice}, suppose that the join of $(u,p)$ and $(v,q)$ is $(a_k,p+k)$ with $p+k=q+j$, so $\rho((u,p))+k=\rho((u,p)\vee(v,q))$. Similar to Proposition \ref{prop:lattice}, we can find the meet of $(u,p)$ and $(v,q)$. If $j\geq k$ then $(u,p)\wedge(v,q)=(b_k,q)$, and if $j\leq k$ then $(u,p)\wedge(v,q)=(a_{i+j-k},p)$. In both cases, we have $\rho((u,p)\wedge(v,q))+k=\rho((v,q))$. This proves the equality and hence modularity.
        
        Symmetry is more straightforward. we will show that the involution $(i,j) \rightarrow (i',j)$ where $i' = 2^{\text{level}(i)}+2^{\text{level}(i)-1}-1-i$ is an automorphism that gives vertical symmetry. This is an involution because $\text{level}(i) = \text{level}(i') = k$ and $i+i' = 2^{k}+2^{k-1}-1$. This is an automorphism because there are only three possible elements covering $(i,j)$: $(2i,j),(2i+1,j)$ and $\left(\left[\dfrac{i}{2}\right],j+1\right)$. The element $(2i,j)$ is mapped to $(2^{\text{level}(i)+1}+2^{\text{level}(i)}-1-2i,j) = (2i'+1,j)$, which covers $(i',j)$. One can check that this also holds for the other two elements analogously.

        Furthermore, we also have the involution $(i,j) \rightarrow (i,j')$ where $j' = (n-1) - \text{level}(i) - j$ is an anti-automorphism that gives horizontal symmetry. This can be proved similarly to the automorphism above.
    \end{proof}

    One may also ask if the lattice is a distributive lattice. Unfortunately, the answer is no because already for the case of $n\geq 4$, we see that the lattice $P_n$ contains the diamond $M_3$, consisting of five elements on three ranks such that there is a bottom element and a top element connected to each of the three elements in the middle.  This is a well-known obstruction to distributivity.
    
    Lastly, it can be seen that every poset $P_n$ can be seen as a sublattice of $P_{\infty}$. Consider the poset $P_{\infty}$ and let $f(i)$ be the number of elements whose rank is $i$, we have the following proposition on $f(i)$.
    
    \begin{proposition} \label{prop:f(i)}
        For all $i>0$, we have
        
        \[ f(i) = \begin{cases} \dfrac{2^{i+1}-1}{3} & \text{if } i \text{ is odd} \\ \dfrac{2^{i+1}-2}{3} & \text{if } i \text{ is even} \end{cases}. \]
    \end{proposition}
    
    \begin{proof}
        The easiest way to see this is that $f(i)$ is the number of nodes whose level has the same parity as $i$. The number of such nodes can actually be written in binary form as a sequence of $i$ digits, starting with 1, and the 1 and 0 digits alternate. This number is the sequence \href{http://oeis.org/A000975}{\em A000975} on OEIS and can be easily calculated using the formula above.
    \end{proof}
    
    \begin{corollary} \label{cor:endgame-fire}
    For all $n > 0$, the number of firing moves in the endgame, i.e. the number of vertices in lattice $P_n$, is $2^n - (n + 1)$.
    \end{corollary}
    
    \begin{proof}
        The easiest way to see this is to look at the waves in the endgame. In the $i$th wave, we fire every node from level 1 to $n-i$, so there are $2^{n-i}-1$ moves in the $i$th wave. Thus, there are $\sum_{i=1}^{n-1}2^{n-i}-1=2^n-(n+1)$ moves in the endgame.
        
        Another way to prove this corollary is to note that the map $\phi$ that sends every move $(i,j)$ to $(i,n-1-$level$(i)-j)$ is an order-reversing $P_n$-isomorphism. Then the result follows from Proposition \ref{prop:f(i)}.
    \end{proof}
    
    \justify
    \section{Further Questions}\label{sec:conclusion}
    
    Two natural questions that arise in our study are
    
    \begin{enumerate}
        \item What are the reachable terminal configurations?
        \item How many reachable terminal configurations are there?
    \end{enumerate}
    
    The answer for the cases $n=1$ and 2 is straightforward. We also know the answer when $n=3$. By Theorem \ref{thm:main-thm} and Proposition \ref{lemma:type2}, we know that nodes 4, 2, 3, 7 have to contain the chips 1, 2, 6, 7 respectively. However, the other nodes do not have to contain a fixed chip, so in Figure \ref{Figure 12}, $(a,b,c)$ can be an arbitrary permutation of $(3,4,5)$. Indeed, we can check that every permutation gives a reachable terminal configuration. Thus, there are 6 reachable terminal configurations in total.
    
    \centering
    \begin{figure}[h!]
        \centering
        \includegraphics[width = 0.5\textwidth]{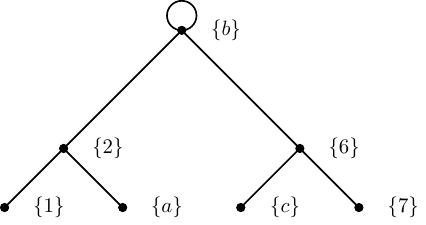}
        \caption{Terminal configurations when $n=3$}
        \label{Figure 12}
    \end{figure}
    
    \justify
    The situation quickly becomes more complicated when $n=4$. Again, by Theorem \ref{thm:main-thm} and Proposition \ref{lemma:type2}, we know that nodes 8, 4, 7, 15 have to contain the chips 1, 2, 14, 15 respectively, but the other 11 nodes do not have to contain a fixed chip. Even though there are additional relations among the nodes, those relations do not decrease the number of reachable configurations down to some small numbers. 
    In fact, as we were finishing this paper, Patrick Liscio coded up the $n=4$ case in python and shared his code with us. He found a total of 36220 possible terminal configurations, which can be found \href{https://drive.google.com/file/d/1b_QkfDahUJuDoF7YhMf_oCU6qNcSY_Sj/view?usp=sharing}{\em here}. In addition, by Proposition \ref{lemma:type1}, with $x_i$'s be the chips in Figure \ref{Figure 13}, we know that $x_1<x_5;x_2,x_3<x_4<x_5;x_2<x_6<x_{10};x_7<x_8<x_{10},x_9;x_7<x_{11}$. Indeed, these relations are all we can have, that is, for every other pair $(x_i,x_j)$, there is a reachable configuration in which $x_i<x_j$ and one in which $x_i>x_j$.
    
    \begin{figure}[h!]
        \centering
        \includegraphics[width = 0.5\textwidth]{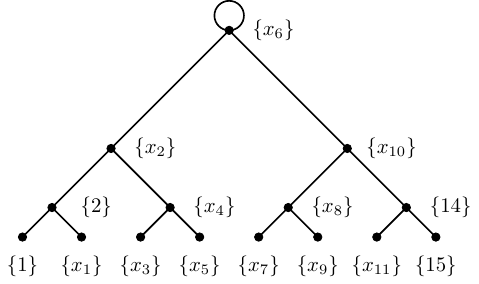}
        \caption{Terminal configurations when $n=4$}
        \label{Figure 13}
    \end{figure} 

    Another question is which terminal configuration are we most likely to get if we fire randomly. In the case $n=3$, we know that we will fire 6 moves, in which 4 are endgame moves. We already know that endgame moves do not affect the terminal configuration, so the terminal configuration is actually determined by the first 2 moves at the root. There are $\binom{7}{3}\binom{5}{3}=350$ possible ways we can fire the first 2 moves. We found that, with $\{a,b,c\}$ being the chips in Figure \ref{Figure 12}, the permutation $(a,b,c)=(3,4,5)$ occurs 216 times, $(3,5,4)$ and $(4,3,5)$ both occur 54 times, $(5,3,4)$ and $(4,5,3)$ both occur 12 times, and $(5,4,3)$ only occur twice. Specifically, to get the $(5,4,3)$ configuration, in the first 2 moves, we have to fire the chips in triplets as $\{1,2,3\}$ and $\{5,6,7\}$. To get the $(5,3,4)$ configuration, we can either fire the triplet $\{5,6,7\}$ then any triplet except $\{1,2,3\}$, or we can first fire any triplet among the chips $1,2,3,4$ except $\{1,2,3\}$ then fire $\{5,6,7\}$, which gives $12$ ways in total. The analogous argument applies for the the $(4,5,3)$ configuration, which also gives $12$ ways in total. To get the $(3,5,4)$ configuration, we have to fire $\{x,y,4\}$ where $x,y,<4$ and not fire $\{1,2,3\}$ or $\{5,6,7\}$, it can be checked that there are $54$ ways in total. Similarly, there are $54$ ways to get the $(4,3,5)$ configuration. Finally, every other way will give the $(3,4,5)$ configuration; thus, the most likely terminal configuration is the one in Figure \ref{Figure 14}.
    
    \centering
    \begin{figure}[h!]
        \centering
        \includegraphics[width = 0.5\textwidth]{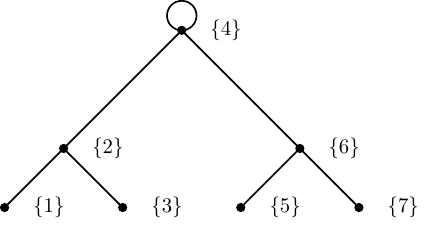}
        \caption{Most likely terminal configuration when $n=3$}
        \label{Figure 14}
    \end{figure}
    
    \justify
    In the $n=4$ case, it takes $23$ firing moves (see Corollary \ref{cor:num-fire}) to get from the initial configuration to a terminal configuration. In light of Corollary \ref{cor:endgame-fire}, $11$ of these moves take place in the endgame and do not affect the resulting terminal configuration. 
    Using the aforementioned python code from Patrick Liscio for the $n=4$ case, the most likely terminal configuration is the one in Figure \ref{Figure 15} with a frequency of approximately $3.846\%$. When $n\geq5$, however, the number of possible terminal configurations grows so fast that we could not get reliable data to estimate relative frequencies of different terminal configurations. Liscio's code also gave the min and max bounds for chip labels at different nodes, as shown in the data in the Appendix. In particular, even in the $n=4$ case, there are restrictions on the possible terminal configurations (i.e. on the values of $x_1$, $x_2$, $\dots$, $x_{11}$) beyond the fifteen inequalities mentioned above due to Propositions \ref{lemma:type1} and \ref{lemma:type2}.  This continues to be an interesting area for future work.
    
    \begin{figure}[h!]
        \centering
        \includegraphics[width = 0.5\textwidth]{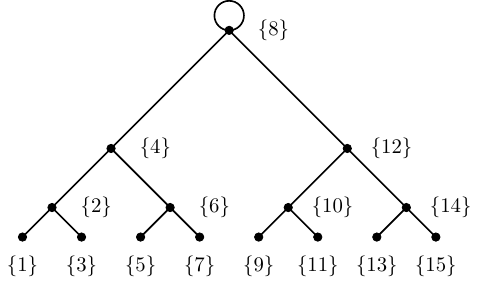}
        \caption{Most likely terminal configuration when $n=4$}
        \label{Figure 15}
    \end{figure} 
    
    It is noticeable that in both cases, the most likely terminal configuration is a binary search tree. Recall from section \ref{sec:counterex} that a \textbf{binary search tree} is a binary tree in which for every node, its chip is larger than \textit{all} chips in its left subtree, and is smaller than \textit{all} chips in its right subtree. For a full binary tree on $n$ levels using the labels $1,2,\dots, 2^n-1$, there is a unique such binary search tree.  We thus make the following conjecture.
    
    \begin{conjecture} \label{con:most-likely}
        For all $n$, the most likely terminal configuration is the unique binary search tree on the complete binary search tree on $2^n-1$ nodes. 
    \end{conjecture}
    
    \textbf{Remark:} One way to view the terminal configuration is to list the chips from left to right. This will turn any terminal configuration to a permutation in $S_{2^n-1}$. For example, the binary tree in Figure \ref{Figure 15} corresponds to the identity permutation in $S_{15}$ (see the Appendix for another example). Thus, another way to phrase Conjecture \ref{con:most-likely} is "For all $n$, the most likely terminal configuration corresponds to the identity permutation in $S_{2^n-1}$."
    
    While it might be tempting to suggest that configurations with fewer inversions occur with higher probability, this is not necessarily true. For example, configurations with two inversions in the middle occur more frequently than those with one inversion towards either side of the tree.

    Another natural question is what happens if we start with any number of chips, not necessarily $2^n-1$. However, by Corollary \ref{cor:pattern}, if we do not start with $2^n-1$ chips, in the terminal configurations, some node has two chips. Thus, some endgame moves will occur when the node has more than three chips. This means that we will not have confluence in the endgame. Hence, many results here will not hold. For example, if we start with $2^n$ chips, we can keep the chip labeled $1$ at the root and never fire it. Hence not only would the first part of our main theorem (Theorem \ref{thm:main-thm}) no longer hold, but the second part of our main theorem regarding the labels of bottom straight left and right descendants would also fail.  It is an interesting open question for future research to determine what restrictions, if any, there are on labelings of terminal configurations when we do not start with $2^n-1$ chips at the root.
    
    \textbf{Postscript:} As we were completing this paper, we learned about work by Hugunin and Roman at an REU at University of Washington \cite{roman_hugunin_2017} that also considered labeled chip-firing on trees, although under different definitions than our set-up.  It would be interesting for future work to compare our two approaches and the different questions that the two projects focused on.





\bibliography{bibliography}
\bibliographystyle{alpha}
	
	\newpage
	
	\appendix
	
	\section{Liscio's Code Results}
	
	How to read:
	
    - First section:
    
    Max/min at [Node]: [Chip]
    
    [Tree in which max/min occurs]\\

    Here [Node] is labeled using the convention of Figure \ref{Figure 2}, i.e. the root is node $1$ and every node $i$ has a left child labeled $2i$ and a right child labeled $2i+1$.
    
    [Chip] is the label of the unique chip at [Node] in the terminal configuration.  Max/min refers to the label of the chip at [Node].\\

    - Second section:
    
    [Tree], [\# Inversions]\\
    
    - Tree: read from left to right
    \\

Ex: [1, 2, 3, 4, 5, 6, 7, 12, 8, 9, 10, 13, 11, 14, 15] means
    
    \hspace{200pt}12
    
    \vspace{10pt}
    
    \hspace{100pt}4\hspace{200pt} 13
    
    \vspace{10pt}
    
    \hspace{50pt}2\hspace{100pt}6\hspace{100pt}9\hspace{100pt}14
    
    \vspace{10pt}
    
    \hspace{25pt}1\hspace{50pt}3\hspace{40pt}5\hspace{50pt}7\hspace{50pt}8\hspace{40pt}10\hspace{50pt}11\hspace{40pt}15\\
    
    RESULT:\\

Number of configurations: 36220\\

Max at 1: 12

[1, 2, 3, 4, 5, 6, 7, 12, 8, 9, 10, 13, 11, 14, 15]

Min at 1: 4

[1, 2, 5, 3, 6, 7, 8, 4, 9, 10, 11, 12, 13, 14, 15]\\

Max at 2: 8

[1, 2, 3, 8, 4, 9, 10, 11, 5, 6, 7, 12, 13, 14, 15]

Min at 2: 3

[1, 2, 4, 3, 5, 6, 7, 8, 9, 10, 11, 12, 13, 14, 15]\\

Max at 3: 13

[1, 2, 3, 4, 5, 6, 7, 8, 9, 10, 11, 13, 12, 14, 15]

Min at 3: 8

[1, 2, 3, 4, 9, 10, 11, 5, 6, 7, 12, 8, 13, 14, 15]\\

Max at 4: 2

[1, 2, 3, 4, 5, 6, 7, 8, 9, 10, 11, 12, 13, 14, 15]

Min at 4: 2

[1, 2, 3, 4, 5, 6, 7, 8, 9, 10, 11, 12, 13, 14, 15]\\

Max at 5: 11

[1, 2, 3, 4, 5, 11, 12, 6, 7, 8, 9, 10, 13, 14, 15]

Min at 5: 5

[1, 2, 6, 3, 4, 5, 8, 7, 9, 10, 11, 12, 13, 14, 15]\\

Max at 6: 11

[1, 2, 3, 4, 5, 6, 7, 9, 8, 11, 12, 13, 10, 14, 15]

Min at 6: 5

[1, 2, 3, 6, 7, 8, 9, 10, 4, 5, 11, 12, 13, 14, 15]\\

Max at 7: 14

[1, 2, 3, 4, 5, 6, 7, 8, 9, 10, 11, 12, 13, 14, 15]

Min at 7: 14

[1, 2, 3, 4, 5, 6, 7, 8, 9, 10, 11, 12, 13, 14, 15]\\

Max at 8: 1

[1, 2, 3, 4, 5, 6, 7, 8, 9, 10, 11, 12, 13, 14, 15]

Min at 8: 1

[1, 2, 3, 4, 5, 6, 7, 8, 9, 10, 11, 12, 13, 14, 15]\\

Max at 9: 10

[1, 2, 10, 3, 4, 5, 12, 6, 7, 8, 9, 11, 13, 14, 15]

Min at 9: 3

[1, 2, 3, 4, 5, 6, 7, 8, 9, 10, 11, 12, 13, 14, 15]\\

Max at 10: 10

[1, 2, 3, 4, 10, 11, 12, 5, 6, 7, 8, 9, 13, 14, 15]

Min at 10: 3

[1, 2, 4, 5, 3, 6, 7, 8, 9, 10, 11, 12, 13, 14, 15]\\

Max at 11: 13

[1, 2, 3, 4, 5, 6, 13, 7, 8, 9, 10, 11, 12, 14, 15]

Min at 11: 7

[1, 2, 3, 4, 5, 6, 7, 8, 9, 10, 11, 12, 13, 14, 15]\\

Max at 12: 9

[1, 2, 3, 4, 5, 6, 7, 8, 9, 10, 11, 12, 13, 14, 15]

Min at 12: 3

[1, 2, 4, 5, 6, 7, 8, 9, 3, 10, 11, 12, 13, 14, 15]\\

Max at 13: 13

[1, 2, 3, 4, 5, 6, 7, 8, 9, 10, 13, 11, 12, 14, 15]

Min at 13: 6

[1, 2, 3, 7, 8, 9, 10, 11, 4, 5, 6, 12, 13, 14, 15]\\

Max at 14: 13

[1, 2, 3, 4, 5, 6, 7, 8, 9, 10, 11, 12, 13, 14, 15]

Min at 14: 6

[1, 2, 3, 5, 7, 8, 9, 10, 4, 11, 12, 13, 6, 14, 15]\\

Max at 15: 15

[1, 2, 3, 4, 5, 6, 7, 8, 9, 10, 11, 12, 13, 14, 15]

Min at 15: 15

[1, 2, 3, 4, 5, 6, 7, 8, 9, 10, 11, 12, 13, 14, 15]\\

SAMPLE OF THE FULL RESULT:

[1, 2, 3, 4, 5, 6, 7, 8, 9, 10, 11, 12, 13, 14, 15], 0

[1, 2, 3, 4, 5, 6, 7, 8, 9, 10, 12, 11, 13, 14, 15], 1

[1, 2, 3, 5, 4, 6, 7, 8, 9, 10, 11, 12, 13, 14, 15], 1

[1, 2, 4, 3, 5, 6, 7, 8, 9, 10, 11, 12, 13, 14, 15], 1

[1, 2, 3, 4, 5, 6, 7, 9, 8, 10, 11, 12, 13, 14, 15], 1

[1, 2, 3, 4, 5, 6, 8, 7, 9, 10, 11, 12, 13, 14, 15], 1

[1, 2, 3, 4, 5, 6, 7, 8, 9, 10, 11, 13, 12, 14, 15], 1

[1, 2, 3, 4, 5, 6, 7, 10, 8, 9, 11, 12, 13, 14, 15], 2

[1, 2, 4, 5, 3, 6, 7, 8, 9, 10, 11, 12, 13, 14, 15], 2

[1, 2, 3, 4, 5, 6, 7, 8, 9, 10, 13, 11, 12, 14, 15], 2

[1, 2, 3, 4, 5, 6, 7, 9, 8, 10, 12, 11, 13, 14, 15], 2

[1, 2, 4, 3, 5, 6, 7, 8, 9, 10, 11, 13, 12, 14, 15], 2

[1, 2, 3, 4, 5, 6, 7, 8, 9, 10, 12, 13, 11, 14, 15], 2

[1, 2, 3, 4, 5, 6, 8, 9, 7, 10, 11, 12, 13, 14, 15], 2

[1, 2, 5, 3, 4, 6, 7, 8, 9, 10, 11, 12, 13, 14, 15], 2

[1, 2, 3, 4, 5, 6, 9, 7, 8, 10, 11, 12, 13, 14, 15], 2

[1, 2, 3, 5, 4, 6, 7, 8, 9, 10, 11, 13, 12, 14, 15], 2

[1, 2, 4, 3, 5, 6, 7, 8, 9, 10, 12, 11, 13, 14, 15], 2

[1, 2, 4, 3, 5, 6, 8, 7, 9, 10, 11, 12, 13, 14, 15], 2

[1, 2, 4, 3, 5, 6, 7, 9, 8, 10, 11, 12, 13, 14, 15], 2

[1, 2, 3, 5, 4, 6, 7, 9, 8, 10, 11, 12, 13, 14, 15], 2

[1, 2, 3, 4, 5, 7, 8, 6, 9, 10, 11, 12, 13, 14, 15], 2

[1, 2, 3, 4, 5, 6, 7, 9, 8, 10, 11, 13, 12, 14, 15], 2

[1, 2, 3, 4, 5, 6, 8, 7, 9, 10, 12, 11, 13, 14, 15], 2

[1, 2, 3, 5, 4, 6, 7, 8, 9, 10, 12, 11, 13, 14, 15], 2

[1, 2, 3, 5, 4, 6, 8, 7, 9, 10, 11, 12, 13, 14, 15], 2

[1, 2, 3, 4, 5, 6, 8, 7, 9, 10, 11, 13, 12, 14, 15], 2

~
$\hspace{15em} \vdots$

~

[1, 2, 9, 5, 4, 10, 13, 8, 3, 6, 12, 11, 7, 14, 15], 25

[1, 2, 4, 5, 10, 11, 13, 9, 3, 7, 8, 12, 6, 14, 15], 25

[1, 2, 9, 4, 5, 10, 13, 8, 3, 7, 12, 11, 6, 14, 15], 25

[1, 2, 10, 5, 4, 8, 13, 11, 3, 6, 9, 12, 7, 14, 15], 25

[1, 2, 10, 4, 8, 9, 13, 7, 3, 5, 6, 11, 12, 14, 15], 25

[1, 2, 9, 4, 7, 10, 13, 5, 3, 8, 12, 11, 6, 14, 15], 25

[1, 2, 8, 6, 4, 11, 13, 7, 3, 5, 12, 10, 9, 14, 15], 25

[1, 2, 8, 4, 7, 11, 13, 5, 3, 9, 12, 10, 6, 14, 15], 25

[1, 2, 9, 6, 4, 10, 13, 7, 3, 5, 12, 11, 8, 14, 15], 25

[1, 2, 8, 4, 5, 11, 13, 9, 3, 7, 12, 10, 6, 14, 15], 25

[1, 2, 7, 6, 4, 11, 13, 9, 3, 5, 12, 10, 8, 14, 15], 25

[1, 2, 8, 5, 4, 11, 13, 9, 3, 6, 12, 10, 7, 14, 15], 25

[1, 2, 10, 6, 4, 7, 13, 11, 3, 5, 9, 12, 8, 14, 15], 25

[1, 2, 10, 5, 4, 9, 13, 8, 3, 6, 11, 12, 7, 14, 15], 25

[1, 2, 10, 6, 4, 9, 13, 7, 3, 5, 11, 12, 8, 14, 15], 25

\end{document}